\documentclass[a4paper,reqno,11pt]{amsart}
\usepackage{verbatim} 
\usepackage{amssymb,amsmath,amsthm}
\usepackage{amsfonts}
\usepackage{array}

\def\ds{\displaystyle}

\def\HH{\overline{H}}
\def\Li{\mbox{Li}}
\def\NN{\mathbb N}

\newtheorem{theorem}{Theorem}
\newtheorem{lemma}{Lemma}
\newtheorem{corollary}{Corollary}

\newtheorem{proposition}{Proposition}

\theoremstyle{definition}

\begin{document}

\title{Congruences concerning Jacobi polynomials and Ap\'ery-like formulae}

\author{Kh.~Hessami Pilehrood}

\author{T.~Hessami Pilehrood}


\address{Department of Mathematics and Statistics, Dalhousie University, Halifax, Nova Scotia, B3H~3J5, Canada}

\email{hessamik@gmail.com, hessamit@gmail.com}

\author{R.~Tauraso}
\address{
Dipartimento di Matematica,
Universit\`a di Roma ``Tor Vergata'', 00133 Roma, Italy
}

\email{tauraso@mat.uniroma2.it}


\date{}


\begin{abstract}
Let $p>5$ be a prime. We prove congruences modulo $p^{3-d}$
for sums of the general form $\sum_{k=0}^{(p-3)/2}\binom{2k}{k}t^k/(2k+1)^{d+1}$
and $\sum_{k=1}^{(p-1)/2}\binom{2k}{k}t^k/k^d$
with $d=0,1$. We also consider the special case $t=(-1)^{d}/16$
of the former sum, where the congruences hold modulo $p^{5-d}$.
\end{abstract}

\maketitle

\section{Introduction}\label{intro}

In proving the irrationality of $\zeta(3),$ Ap\'ery \cite{Ap:79}
mentioned the formulae
$$
\sum_{k=1}^{\infty}\frac{1}{k^2}\binom{2k}{k}^{-1}=\frac{1}{3}\,\zeta(2), \qquad
\sum_{k=1}^{\infty}\frac{(-1)^k}{k^3}\binom{2k}{k}^{-1}=-\frac{2}{5}\,\zeta(3).
$$
Later Koecher \cite{Ko:80}, and Leshchiner \cite{Le:81}, found several
analogous results for other $\zeta(r).$ In particular, in \cite{Le:81},
by the use of some combinatorial identities it was shown that
\begin{align}
\sum_{k=0}^{\infty}\frac{\binom{2k}{k}}{16^k(2k+1)}&=\frac{4}{3}\sum_{k=0}^{\infty}
\frac{(-1)^k}{2k+1}=\frac{\pi}{3}, \label{id1} \\
\sum_{k=0}^{\infty}\frac{\binom{2k}{k}}{(-16)^k(2k+1)^2}&=\frac{4}{5}\sum_{k=0}^{\infty}
\frac{1}{(2k+1)^2}=\frac{\pi^2}{10}. \label{id2}
\end{align}
The above evaluations are related to the power series of $\arcsin(z),$
\begin{equation} \label{arcsin}
\arcsin(z)=\sum_{k=0}^{\infty}\frac{\binom{2k}{k} z^{2k+1}}{4^k(2k+1)}, \qquad |z|\le 1,
\end{equation}
and its integral $\int_0^z\frac{\arcsin(y)}{y}\,dy,$ which can be expressed in closed form
(see \cite[Theorem~5.1]{BC:07}) as
\begin{equation} \label{intarcsin}
\begin{split}
\sum_{k=0}^{\infty}\frac{\binom{2k}{k} z^{2k+1}}{4^k(2k+1)^2}
&=-\frac{i}{2}\Li_2((\sqrt{1-z^2}+iz)^2)-\frac{i}{2}\arcsin^2(z) \\
&+\arcsin(z)\log(2z^2-2iz\sqrt{1-z^2})
+\frac{i}{2}\zeta(2), \quad |z|\le 1,
\end{split}
\end{equation}
where $\Li_r(z)=\sum_{k=1}^{\infty}z^k/k^r$ is the polylogarithmic function.

Recently, Z.~W.~Sun \cite{Sunzw:11}, \cite[p.~27]{Sunzw:10} pointed out
that series (\ref{arcsin}), (\ref{intarcsin}) for some
special values of variable admit nice finite $p$-analogues. He obtained the congruences for partial sums
$$
\sum_{k=0}^{(p-3)/2}\frac{\binom{2k}{k} t^k}{(2k+1)^{d+1}} \pmod{p^{2-d}}
$$
for $t=1/4, 1/8, 1/16, 3/16$ if $d=0,$ and $t=1/4$ if $d=1.$

In this paper, we refine Sun's congruences to a higher power of $p$ and obtain polynomial congruences for
general sums of the form
\begin{equation} \label{first}
\sum_{k=0}^{(p-3)/2}\frac{\binom{2k}{k} t^k}{(2k+1)^{d+1}} \qquad\text{and}\qquad
\sum_{k=1}^{(p-1)/2}\frac{\binom{2k}{k}t^k}{k^d} \pmod{p^{3-d}},
\end{equation}
with $d=0, 1.$
Our approach is based on application of functional properties of Jacobi polynomials
$P_n^{(\alpha,\beta)}(x),$ which are defined in terms of the Gauss hypergeometric function
$$
F(a,b;c;z)=\sum_{k=0}^{\infty}\frac{(a)_k(b)_k}{k! (c)_k}z^k,
$$
where $(a)_0=1,$ $(a)_k=a(a+1)\cdots(a+k-1),$ $k\ge 1,$ is the Pochhammer symbol, as
\begin{equation}
P_n^{(\alpha,\beta)}(x)=\frac{(\alpha+1)_n}{n!}\,F(-n,n+\alpha+\beta+1;\alpha+1;(1-x)/2),
\quad \alpha, \beta>-1.
\label{eq01}
\end{equation}
They satisfy the three term recurrence relation \cite[Section 4.5]{Sz:59}
\begin{equation}
\begin{split}
2(n+1)&(n+\alpha+\beta+1)(2n+\alpha+\beta)P_{n+1}^{(\alpha,\beta)}(x) \\
&=[(2n+\alpha+\beta+1)(\alpha^2-\beta^2)+(2n+\alpha+\beta)_3x]P_n^{(\alpha,\beta)}(x) \\
&\qquad-2(n+\alpha)(n+\beta)(2n+\alpha+\beta+2)P_{n-1}^{(\alpha,\beta)}(x)
\label{eq02}
\end{split}
\end{equation}
with the initial conditions $P_0^{(\alpha,\beta)}(x)=1,$ $P_1^{(\alpha,\beta)}(x)=(x(\alpha+\beta+2)+\alpha-\beta)/2.$

Some special cases of the Jacobi  polynomials have already been used for proving many remarkable
congruences containing central binomial sums. Thus, for example, the short proof of the
Rodriguez-Villegas famous congruence \cite{RV:03}
\begin{equation}
\sum_{k=0}^{(p-1)/2}\frac{\binom{2k}{k}^2}{16^k}\equiv (-1)^{\frac{p-1}{2}} \pmod{p^2},
\label{eq02.5}
\end{equation}
first confirmed by Mortenson~\cite{Mo:03} via the $p$-adic $\Gamma$-function
and the Gross-Koblitz formula, was given by Z.~H.~Sun~\cite{Sunzh:11} with the help of
Legendre polynomials $P_n(x),$ which are a special case of the Jacobi polynomials: $P_n(x)=P_n^{(0,0)}(x).$
The classical Chebyshev polynomials of the first and second kind, $T_n(x)$ and $U_n(x),$
corresponding to the case $\alpha=\beta=-1/2$ and $\alpha=\beta=1/2$ in (\ref{eq01}), respectively,
\begin{equation}
T_n(x)=\frac{n!}{(1/2)_n}\,P_n^{(-1/2,-1/2)}(x), \quad
U_n(x)=\frac{(n+1)!}{(3/2)_n}\,P_n^{(1/2, 1/2)}(x),
\label{eq03}
\end{equation}
were applied in many works on congruences (see, for example,
\cite{MT:11, Sunzw3, Sunzw:11, ST:10}).
In \cite{MT:11}, S.~Mattarei and R.~Tauraso used the classical Chebyshev polynomials
\begin{equation}
u_0(x)=0, \quad u_n(x)=U_{n-1}(x/2) \quad (n\ge 1), \quad  \quad v_n(x)=2T_n(x/2) \quad (n\ge 0)
\label{eq06}
\end{equation}
to obtain congruences for the finite sums
$$
\sum_{k=1}^{p-1}\frac{t^k}{k^d\binom{2k}{k}} \pmod{p^2}, \qquad
\sum_{k=1}^{p-1}\frac{t^{p-k}H_k(2)\binom{2k}{k}}{k^d} \pmod{p},
$$
where $H_k(m)=\sum_{j=1}^k1/j^m$ and $d=0,1,2,$ which
are related to  series expansions of even powers of $\arcsin$ and
reveal some connections with infinite series
for zeta values (more on zeta series and congruences related to the above sums may be found in \cite{Ta:10, HH:11}).

Note that the sequences $u_n(x)$ and $v_n(x)$ given by (\ref{eq06}) belong to a class of the so called Lucas sequences
$u_n(x,y)$ and $v_n(x,y)$ which are defined by the recurrence relations
\begin{align*}
u_0(x,y)=0, \quad & u_1(x,y)=1, & \text{and}\quad & u_n(x,y)=xu_{n-1}(x,y)-yu_{n-2}(x,y) & \text{for}\,\, n>1, \\
v_0(x,y)=2, \quad & v_1(x,y)=x, & \text{and}\quad & v_n(x,y)=xv_{n-1}(x,y)-yv_{n-2}(x,y) & \text{for}\,\, n>1.
\end{align*}
Namely, one has $u_n(x)=u_n(x,1)$ and $v_n(x)=v_n(x,1).$

The paper is organized as follows. In Section \ref{Section2}, we prove some results concerning divisibility properties
of multiple harmonic sums. In Section~\ref{Section3},
 we consider two ``mixed'' cases of Jacobi polynomials, $P_n^{(1/2, -1/2)}(x)$ and
$P_n^{(-1/2,/2)}(x),$ and study their properties with application to congruences.
In Section~\ref{Section4}, we get some numerical congruences.
In  Section~\ref{Section5}, by revisiting the combinatorial proof (due to D.~Zagier)
presented in \cite[Section 5]{Le:81}, we obtain the finite versions of identities (\ref{id1}) and (\ref{id2}),
which allow us to improve significantly congruences for the first sum in (\ref{first}) in the special case $t=(-1)^d/16.$


\section{Results concerning multiple harmonic sums}
\label{Section2}

We define the {\it multiple harmonic sum} as
$$H_n(a_1,a_2,\dots,a_r)=
\sum_{0< k_1<k_2<\dots<k_r\leq n}\; \frac{1}{k_1^{a_1}k_2^{a_2}\cdots k_r^{a_r}},$$
where $n\geq r>0$ and $(a_1,a_2,\dots,a_r)\in (\NN^*)^r$.
We also introduce the multiple sum
$$\HH_n(a_1,a_2,\dots,a_r)=
\sum_{0\leq k_1<k_2<\dots<k_r<n}\; \frac{1}{(2k_1+1)^{a_1}(2k_2+1)^{a_2}\cdots (2k_r+1)^{a_r}},$$
with $(a_1,a_2,\dots,a_r)\in (\NN^*)^r$. For the sake of brevity, if $a_1=a_2=\dots=a_r=a$, we write $H_n(\{a\}^r)$ and $\HH_n(\{a\}^r)$.
Note that both $H_n$ and $\HH_n$ satisfy the shuffle product property:
\begin{align*}
H_n(a)H_n(b)&= H_n(a,b)+H_n(b,a)+H_n(a+b),\\
\HH_n(a)\HH_n(b)&= \HH_n(a,b)+\HH_n(b,a)+\HH_n(a+b)
\end{align*}
for any positive integers $n,a,b$.

The values of many harmonic sums modulo a power of prime $p$ are well known.
Here is a list of results that we will need later.

\begin{enumerate}

\item[(i)] (\cite[Theorem~5.1]{Sunzh:00}) for any prime $p>r+2$ we have
\begin{align*}
H_{p-1}(r)\equiv
\begin{cases}
-\frac{r(r+1)}{2(r+2)}\,p^2\,B_{p-r-2}  \pmod{p^3} &\mbox{if $r$ is odd,}\\
\frac{r}{r+1}\,p\,B_{p-r-1}  \pmod{p^2} &\mbox{if $r$ is even;}
\end{cases}
\end{align*}

\item[(ii)] (\cite[Theorem~3.1]{Zh:06}) for positive integers $r,s$  and for any prime $p>r+s$, we have
$$H_{p-1}(r,s)\equiv \frac{(-1)^s}{r+s}\binom{r+s}{s} \,B_{p-r-s}\pmod{p};$$

\item[(iii)] (\cite[Theorem~3.5]{Zh:06}) for positive integers $r,s,t$ and for any prime $p>r+s+t$ such that
$r+s+t$ is odd, we have
$$H_{p-1}(r,s,t)\equiv \frac{1}{2(r+s+t)}\left((-1)^r\binom{r+s+t}{r}-(-1)^t\binom{r+s+t}{t}\right) \,B_{p-r-s-t}\pmod{p};$$

\item[(iv)] (\cite[Theorem~2.1]{Ta:10}) for any prime $p>5$,
$$H_{p-1}(1)\equiv -\frac{1}{2}\,p H_{p-1}(2)-\frac{1}{6}\,p^2 H_{p-1}(3)\pmod{p^5};$$

\item[(v)] (\cite[Lemma~3]{HH:11}) for any prime $p>5$,
$$H_{p-1}(1,2)\equiv -3\,\frac{H_{p-1}(1)}{p^2}+\frac{1}{2}\,p^2B_{p-5}\pmod{p^3};$$

\item[(vi)] (\cite[Theorem~5.2]{Sunzh:00}) for any prime $p>r+4$ we have
\begin{align*}
H_{\frac{p-1}{2}}(r)\equiv
\begin{cases}
-2q_p(2)+p q^2_p(2)-p^2\left(\frac{2}{3}\,q^3_p(2)
+\frac{7}{12}\,B_{p-3}\right)  \pmod{p^3} &\mbox{if $r=1$,}\\[6pt]
\frac{r(2^{r+1}-1)}{2(r+1)}\,p\,B_{p-r-1}  \pmod{p^2} &\mbox{if $r$ is even,}\\[6pt]
-\frac{2^r-2}{r}\,B_{p-r}  \pmod{p} &\mbox{if $r>1$ is odd,}
\end{cases}
\end{align*}
where $q_p(a)=(a^{p-1}-1)/p$ is the so-called {\sl Fermat quotient}.
\end{enumerate}

The next three results are useful congruences involving $H_{\frac{p-1}{2}}(r)$ and $H_{\frac{p-1}{2}}(r,s)$.

\begin{lemma}\label{Ldue} Let $r,a$ be positive integers. Then for any prime $p>r+2,$
\begin{align}\label{C1}
H_{p-1}(r)&\equiv H_{\frac{p-1}{2}}(r)+(-1)^r\sum_{k=0}^a\binom{r-1+k}{k} H_{\frac{p-1}{2}}(r+k)p^k\pmod{p^{a+1}}.
\end{align}
Moreover, if $r,s$ are positive integers such that $r+s$ is odd, then for any prime $p>r+s,$
\begin{equation}\label{C2}
H_{\frac{p-1}{2}}(r,s)\equiv {B_{p-r-s}\over 2(r+s)}\left((-1)^s{r+s \choose s}+2^{r+s}-2\right)\pmod{p}.
\end{equation}
\end{lemma}
\begin{proof} The congruence \eqref{C1} follows immediately from the identity
$$H_{{p-1}}(r)=H_{{p-1\over 2}}(r)+(-1)^r\sum_{k=1}^{\frac{p-1}{2}}\frac{1}{k^r(1-p/k)^r}.$$
Now we show \eqref{C2}:
\begin{align*}
H_{{p-1}}(r,s)&=\!\!\!\!
\sum_{1\leq i<j\leq {p-1\over 2}}{1\over i^r j^s}
+\sum_{1\leq i,j\leq {p-1\over 2}}{1\over i^r (p-j)^s}
+\sum_{1\leq j<i\leq {p-1\over 2}}{1\over (p-i)^r (p-j)^s}\\
&\equiv
H_{{p-1\over 2}}(r,s)+(-1)^s\,H_{{p-1\over 2}}(r)H_{{p-1\over 2}}(s)+(-1)^{r+s}H_{{p-1\over 2}}(s,r)\\
&\equiv
H_{{p-1\over 2}}(r,s)-H_{{p-1\over 2}}(s,r) \pmod{p}.
\end{align*}
By the shuffle product property
$$H_{{p-1\over 2}}(r,s)+H_{{p-1\over 2}}(s,r)+H_{{p-1\over 2}}(r+s)=H_{{p-1\over 2}}(r)H_{{p-1\over 2}}(s)\equiv 0\pmod{p}.$$
Hence
$$2H_{{p-1\over 2}}(r,s)\equiv H_{p-1}(r,s)-H_{{p-1\over 2}}(r+s)\pmod{p}$$
and by applying (ii) and (vi) we obtain \eqref{C2}.
\end{proof}

\begin{theorem}\label{T21} For any prime $p>2,$
\begin{equation}\label{Wmezzo}
H_{{p-1\over 2}}(2)+{7\over 6}\, p\,H_{{p-1\over 2}}(3) +{5\over 8}\, p^2\,H_{{p-1\over 2}}(4)\equiv 0
\pmod{p^4}.
\end{equation}
\end{theorem}
\begin{proof}  Let $m=\varphi(p^4)=p^3(p-1)$ and let $B_n(x)$ be the $n$th Bernoulli polynomial.
For $r=2,3,4$, Faulhaber's formula implies
\begin{align*}
\sum_{k=1}^{(p-1)/2}k^{m-r}&={B_{m-r+1}\left({p+1\over 2}\right)-B_{m-r+1}\over m-r+1}\\
&= {B_{m-r+1}\left({p+1\over 2}\right)-B_{m-r+1}\left({1\over 2}\right)\over m-r+1}+{B_{m-r+1}\left({1\over 2}\right)-B_{m-r+1}\over m-r+1}\\
&= \sum_{k=r}^m{m-r\choose k-r}{B_{m-k}\left({1\over 2}\right)\over k-r+1}\cdot \left({p\over 2}\right)^{k-r+1}
+{B_{m-r+1}\left({1\over 2}\right)-B_{m-r+1}\over m-r+1}.
\end{align*}
By Euler's theorem, $H_{{p-1\over 2}}(r)\equiv \sum_{k=1}^{(p-1)/2}k^{m-r}\pmod{p^4}$.
Since $m$ is even, $B_{m-k}=0$ when $m-k>1$ and $k$ is odd. Moreover, $pB_{m-k}$ is $p$-integral
by von Staudt-Clausen theorem and
$$B_{m-k}\left({1\over 2}\right)=(2^{1-m+k}-1)B_{m-k}\equiv (2^{1+k}-1)B_{m-k}\pmod{p^4}.$$
Hence
\begin{align*}
\sum_{r=2}^4 \alpha_r\, p^{r-2} H_{{p-1\over 2}}(r)&\equiv
{p\over 2}\left(\alpha_2 -{6\over 7}\,\alpha_3\right)B_{m-2}\left({1\over 2}\right)\\
&\qquad +{p^3\over 8}\left(\alpha_2 -3\,\alpha_3+4\alpha_4\right)B_{m-4}\left({1\over 2}\right)
\pmod{p^4}.
\end{align*}
The right-hand side vanishes if we let $\alpha_2=1$, $\alpha_3=7/6$, and $\alpha_4=5/8$.
\end{proof}

\begin{corollary} For any prime $p>5,$
\begin{align}
\label{Hp2}
H_{{p-1}}(2)&\equiv -2\,{H_{{p-1}}(1)\over p}+{2\over 5}\,p^3 B_{p-5}\pmod{p^4},\\
\label{H2}
H_{{p-1\over 2}}(2)&\equiv -7\,{H_{{p-1}}(1)\over p}+{17\over 10}\,p^3 B_{p-5}\pmod{p^4},\\
\label{H3}
H_{{p-1\over 2}}(3)&\equiv 6\,{H_{{p-1}}(1)\over p^2}-{81\over 10}\,p^2 B_{p-5}\pmod{p^3},\\
\label{unoduepiuunotre}
H_{{p-1\over 2}}(1,2)+p\,H_{{p-1\over 2}}(1,3)
&\equiv -{9\over 2}\,{H_{p-1}(1)\over p^2} -{49\over 20}\,\, p^2 B_{p-5}
\pmod{p^{3}}.
\end{align}
\end{corollary}
\begin{proof} The congruence \eqref{Hp2} follows from (iv) and (i).
By \eqref{Wmezzo} and (vi) we have
$$H_{{p-1\over 2}}(2)+{7\over 6}pH_{{p-1\over 2}}(3)+{31\over 4}\,B_{p-5}\,p^3\equiv 0\pmod{p^4}.$$
From \eqref{C1} and (vi), we deduce
$$H_{p-1}(2)\equiv 2H_{{p-1\over 2}}(2)+2pH_{{p-1\over 2}}(3)+{66\over 5}\,B_{p-5}\,p^3\pmod{p^4}.$$
By solving these two congruences with respect to $H_{{p-1\over 2}}(2)$ and $H_{{p-1\over 2}}(3)$, and by using \eqref{Hp2}, we get the result.
Now we show \eqref{unoduepiuunotre}.  For $n=(p-1)/2,$ we consider the identity
$$H_{p-1}(1,2)=H_n(1,2)+H_n(1)\sum_{j=1}^n{1\over (p-j)^2} +\sum_{1\leq j<i\leq n}{1\over (p-i)(p-j)^2}.$$
Then by expanding the sums like in Lemma~\ref{Ldue}, we get
\begin{align*}
H_{p-1}(1,2)&\equiv H_n(1,2)+H_n(1)(H_n(2)+2p\,H_n(3)+3p^2\,H_n(4))-H_n(2,1)\\
&\qquad -p\,(2H_n(3,1)+H_n(2,2))-p^2\,(3H_n(4,1)+2H_n(3,2)+H_n(2,3))
\pmod{p^{3}}.
\end{align*}
By applying the shuffle product property to $H_n(1)H_n(2)$ and $H_n(1)H_n(3)$,
and by using (i), (ii) and (vi), we obtain
$$
H_n(1,2)+p\,H_n(3,1)\equiv {1\over 2}\,H_{p-1}(1,2)-{1\over 2}\,H_n(3)-{27\over 4}\,p^2 B_{p-5}
\pmod{p^{3}}.$$
Thus the proof of \eqref{unoduepiuunotre} is complete as soon as we apply (v) and \eqref{H3}.
\end{proof}

The relations stated in the next lemma
allow us to determine $\HH_{\frac{p-1}{2}}(r)$ and $\HH_{\frac{p-1}{2}}(r,s)$
in terms of the multiple harmonic sums.

\begin{lemma}\label{lemmaHH}
For any positive integers $n,r$ we have
\begin{equation*}
\HH_n(r)=H_{2n}(r)-{H_n(r)\over 2^r}.
\end{equation*}
Moreover, for any odd prime $p$ and for any  $r,s\geq 1$, the following congruence holds modulo $p^3$
\begin{align*}
\HH_{\frac{p-1}{2}}(r,s)
&\equiv {1\over(-2)^{r+s}}\left(H_{\frac{p-1}{2}}(s,r)+{p\over 2}\left(r H_{\frac{p-1}{2}}(s,r+1)+
s H_{\frac{p-1}{2}}(s+1,r)\right)\right.\\
&\left.+{p^2\over 4}\left({r+1\choose 2}
H_{\frac{p-1}{2}}(s,r+2)
+rs H_{\frac{p-1}{2}}(s+1,r+1)
+{s+1\choose 2} H_{\frac{p-1}{2}}(s+2,r)\right)\right).
\end{align*}
\end{lemma}
\begin{proof} It is easily seen that
$$H_{2n}(r)=\sum_{k=0}^{n-1}\frac{1}{(2k+1)^r}+
\sum_{k=1}^{n}\frac{1}{(2k)^r}=\HH_n(r)+{H_n(r)\over 2^r}.$$
The required congruence follows by expanding with respect to the powers of $p$ the identity
$$
\HH_n(r,s)=\sum_{0<j<i\leq n}{1\over (2(n-i)+1)^r(2(n-j)+1)^s}=
{1\over(-2)^{r+s}}\sum_{0<j<i\leq n}{1\over j^s i^r\left(1-{p\over 2i}\right)^r\left(1-{p\over 2j}\right)^s}$$
where $n=(p-1)/2$.
\end{proof}

\begin{lemma} \label{l3} For any prime $p>5,$
\begin{align}\nonumber
2(-1)^{{p-1\over 2}}\sum_{k=0}^{{p-3\over 2}}{(-1)^k\over 2k+1}
&\equiv   \HH_{{p-1\over 2}}(1) -p\HH_{{p-1\over 2}}(2)-p^2\HH_{{p-1\over 2}}(2,1)\\\label{alts}
&\qquad+p^3\HH_{{p-1\over 2}}(2,2) +p^4\HH_{{p-1\over 2}}(2,2,1) \pmod{p^{5}}.
\end{align}
\end{lemma}
\begin{proof} The following identities can be easily proved by the WZ method:
$$\sum_{k=0}^n{(-16)^k{n+k\choose 2k}\over (2k+1){2k\choose k}}
=2(-1)^{n}\sum_{k=0}^{n-1}{(-1)^k\over 2k+1}+{1\over 2n+1}\,
\quad\mbox{and}\quad
\sum_{k=0}^n{(-16)^k{n+k\choose 2k}\over (2k+1)^2{2k\choose k}}={1\over (2n+1)^2}.$$
Let $n=(p-1)/2$. Notice that for $k\in\{0,1,\ldots,n\}$,
\begin{equation}\label{CCC}
{(-16)^k{n+k\choose 2k}\over {2k\choose k}}=\prod_{j=0}^{k-1}\left(1-{p^2\over (2j+1)^2}\right)=
\sum_{j=0}^{k} (-1)^jp^{2j}\,\HH_k(\{2\}^j).
\end{equation}
Thus, the two identities yield
\begin{align}
\HH_n(1)-p^2\HH_n(2,1)+p^4\HH_n(2,2,1)&\equiv 2(-1)^{n}\sum_{k=0}^{n-1}{(-1)^k\over 2k+1}+{1-(-1)^n 4^{p-1}{2n\choose n}^{-1}\over p}\pmod{p^{6}},
\nonumber \\
\HH_n(2)-p^2\HH_n(2,2)+p^4\HH_n(2,2,2)&\equiv {1-(-1)^n 4^{p-1}{2n\choose n}^{-1}\over p^2}\pmod{p^{6}}, \label{eq20.55}
\end{align}
and, by subtracting $p$ times the second congruence from the first one, we get \eqref{alts}.
\end{proof}
It is interesting to note that starting from formula \eqref{eq20.55} and then by using
Lemma~\ref{lemmaHH}, \eqref{C2}, \eqref{Hp2}, \eqref{H2}, (i) and (vi), we easily get a generalization of Morley's congruence~\cite{Ca:53}.
\begin{corollary}
For any prime $p>5,$
$$
\frac{(-1)^{(p-1)/2}}{4^{p-1}}\binom{p-1}{\frac{p-1}{2}}\equiv 1-\frac{1}{4}\, p\, H_{p-1}(1)-
\frac{1}{80}\, p^5 B_{p-5} \pmod{p^6}.
$$
\end{corollary}

\section{Polynomial congruences}

\label{Section3}

For a non-negative integer $n$ we define the sequence
\begin{equation}
w_n(x):=(2n+1)F(-n,n+1;3/2;(1-x)/2)=\frac{n!}{(1/2)_n} P_n^{(1/2, -1/2)}(x).
\label{eq07}
\end{equation}
Using  recurrence relation (\ref{eq02}) it is easy to show that the sequence $w_n(x)$
satisfies the second order linear recurrence with constant coefficients:
$$
w_{n+1}(x)=2xw_n(x)-w_{n-1}(x)
$$
and initial conditions $w_0(x)=1,$ $w_1(x)=1+2x$.
This implies that $w_n(x)$ has the generating function
$$
W_x(z)=\sum_{n=0}^{\infty} w_n(x)z^n=\frac{1+z}{1-2xz+z^2},
$$
which yields
\begin{equation}
w_n(x)=u_{n+1}(2x)+u_n(2x)
\label{eq08b}
\end{equation}
and
\begin{equation}
w_n(x)=\begin{cases}\ds
\frac{(\alpha+1)\alpha^n-(\alpha^{-1}+1)\alpha^{-n}}{\alpha-\alpha^{-1}}
& \quad \mbox{if $x \neq \pm 1$}, \\[3pt]
2n+1 & \quad \mbox{if $x=1$}, \\
(-1)^n & \quad  \mbox{if $x=-1$},
\end{cases}
\label{eq08}
\end{equation}
where $\alpha=x+\sqrt{x^2-1}$. Note that for   $x\in (-1,1)$
we also have the alternative representation:
\begin{equation} \label{eq23.55}
w_n(x)=\cos(n\arccos x)+\frac{x+1}{\sqrt{1-x^2}}\sin(n\arccos x).
\end{equation}

\begin{lemma} Let $p>3$ be a prime and $t=a/b,$ where $a, b$ are integers coprime to $p.$ Then
\begin{align}
&\sum_{k=0}^{(p-3)/2}\frac{\binom{2k}{k}t^k}{2k+1}\,\HH_k(2)
\equiv \frac{1}{64}\left(\frac{-1}{\, t}\right)^{\frac{p+1}{2}}\sum_{k=1}^{p-1}
\frac{v_k(2-16t)}{k^3} \pmod{p}, \label{A2} \\
&\sum_{k=0}^{(p-1)/2}t^k\binom{2k}{k}\HH_k(2)
\equiv \frac{1}{2}\left(
\frac{-1}{\, t}\right)^{\frac{p-1}{2}}\sum_{k=1}^{p-1}
\frac{u_k(2-16t)}{k^2} \pmod{p}. \label{A3}
\end{align}
\end{lemma}
\begin{proof}
Let $n=(p-1)/2$. We reverse the order of summation over $k$:
$$
\sum_{k=0}^{n-1}\frac{\binom{2k}{k} t^k}{2k+1} \,\HH_k(2)
=\sum_{k=1}^{n}\frac{\binom{2n-2k}{n-k} t^{n-k}}{p-2k}\,\HH_{n-k}(2)
$$
and note that by \eqref{CCC}, for $k\in\{0,1,\ldots,n\},$
\begin{align*}
\binom{2n-2k}{n-k}&=\frac{(p-1-2k)(p-2-2k)\cdots(p-(n+k))}{(n-k)!} \\
&\equiv (-1)^{n-k}\frac{(2k+1)(2k+2)\cdots (n+k)}{(n-k)!}\\
&=(-1)^{n-k}\binom{n+k}{2k}\equiv (-1)^n\frac{\binom{2k}{k}}{16^k} \pmod{p}
\end{align*}
and by (vi)
$$
\HH_{n-k}(2)=\sum_{j=k+1}^{n}\frac{1}{(2(n-j)+1)^2}=
\sum_{j=k+1}^{n}\frac{1}{(p-2j)^2}\equiv
\frac{H_{n}(2)-H_k(2)}{4}\equiv -\frac{H_k(2)}{4} \pmod{p}.
$$
Hence we obtain
$$\sum_{k=0}^{n-1}\frac{\binom{2k}{k} t^k}{2k+1} \,\HH_k(2)
\equiv
\frac{(-t)^{n}}{8}\sum_{k=1}^{n}\frac{\binom{2k}{k}H_k(2)}{k (16t)^k}
\equiv \frac{(-t)^{n}}{8}\sum_{k=1}^{p-1}\frac{\binom{2k}{k}H_k(2)}
{k (16t)^k} \pmod{p}.
$$
Similarly, we have
$$
\sum_{k=0}^nt^k\binom{2k}{k}\HH_k(2)\equiv
-\frac{(-t)^n}{4}\sum_{k=1}^{p-1}\frac{\binom{2k}{k}H_k(2)}{(16t)^k} \pmod{p}.
$$
Now the desired congruences follow from the fact that \cite[Section~6]{MT:11}:
$$\sum_{k=1}^{p-1}\frac{\binom{2k}{k} t^{p-k}H_k(2)}{k}
\equiv -2\sum_{k=1}^{p-1}
\frac{v_k(2-t)}{k^3} \pmod{p}
$$
and
$$
\sum_{k=1}^{p-1}\binom{2k}{k}t^{p-k}H_k(2)\equiv
-2t\sum_{k=1}^{p-1}\frac{u_k(2-t)}{k^2} \pmod{p}.
$$
\end{proof}
The next theorem gives us polynomial congruences for sums (\ref{first}) in the case $d=0.$
Among previous results on the second sum, we mention work \cite{Sunzwfib} where Z.~W.~Sun  determined the value of
$\sum_{k=0}^{(p-1)/2}\frac{\binom{2k}{k}}{m^k}$  $\pmod{p^2}$ for any integer $m\not\equiv 0\pmod{p}$
in terms of the Lucas sequences $u_{p\pm 1}(4,m).$
\begin{theorem} \label{t1}
Let $p$ be an odd prime and $t=a/b,$ where $a, b$ are integers coprime to $p.$ Then the following congruences are true:
\begin{align*}
&\sum_{k=0}^{(p-3)/2}\frac{\binom{2k}{k}t^k}{2k+1}\equiv
\frac{w_{\frac{p-1}{2}}(1-8t)-(-16t)^{\frac{p-1}{2}}}{p}
+\frac{p^2}{64}\left(\frac{-1}{\,t}\right)^{\frac{p+1}{2}}\sum_{k=1}^{p-1}
\frac{v_k(2-16t)}{k^3} \pmod{p^3}, \\
(-1)^{\frac{p-1}{2}}&\sum_{k=0}^{(p-1)/2}\binom{2k}{k}t^k\equiv
w_{\frac{p-1}{2}}(8t-1)+\frac{p^2}{2t^{\frac{p-1}{2}}}\sum_{k=1}^{p-1}
\frac{u_k(2-16t)}{k^2} \pmod{p^3}.
\end{align*}
\end{theorem}
\begin{proof}
Setting $n=(p-1)/2$ in (\ref{eq07})  we get
$$
w_{n}(x)=p\cdot F\left(\frac{1-p}{2}, \frac{1+p}{2}; \frac{3}{2}; \frac{1-x}{2}\right).
$$
Observing that
\begin{equation}
\left(\frac{1-p}{2}\right)_k\left(\frac{1+p}{2}\right)_k=
\frac{(2k)!^2}{16^k k!^2}
\prod_{j=0}^{k-1}
\left(1-\frac{p^2}{(2j+1)^2}\right) \quad\text{and}\quad
\left(\frac{3}{2}\right)_k=\frac{(2k+1)!}{4^k k!},
\label{eq09}
\end{equation}
we get the following identity:
\begin{equation}
w_{n}(x)=p\sum_{k=0}^{n}\frac{\binom{2k}{k}}{2k+1}\left(\frac{1-x}{8}\right)^k
\prod_{j=0}^{k-1}\left(1-\frac{p^2}{(2j+1)^2}\right).
\label{eq10}
\end{equation}
Putting $t=(1-x)/8$ in (\ref{eq10}) we have
$$
w_{n}(1-8t)= p\sum_{k=0}^{n-1}\frac{\binom{2k}{k} t^k}{2k+1}
\sum_{j=0}^{k} (-1)^jp^{2j}\,\HH_k(\{2\}^j)+
\binom{2n}{n}t^{n}\prod_{j=0}^{n-1}\left(1-\frac{p^2}{(2j+1)^2}\right).
$$
Using \eqref{CCC} for $k=n$ we obtain
\begin{equation}\label{A}
\frac{w_{n}(1-8t)-(-16t)^{n}}{p}= \sum_{k=0}^{n-1}\frac{\binom{2k}{k}t^k}{2k+1}
\sum_{j=0}^{k} (-1)^jp^{2j}\,\HH_k(\{2\}^j).
\end{equation}
Lastly, by considering the above equality   modulo $p^3$
and by using \eqref{A2}, we get the first congruence of the theorem.
To prove the second one, we apply the well-known symmetry property of the Jacobi polynomials
$$
P_n^{(\alpha, \beta)}(x)=(-1)^n P_n^{(\beta, \alpha)}(-x).
$$
Then from the definition of $w_n(x)$ we get
\begin{equation}
w_{\frac{p-1}{2}}(x)=\frac{(-1)^{\frac{p-1}{2}}(\frac{p-1}{2})!}{(1/2)_{\frac{p-1}{2}}}\,
P_{\frac{p-1}{2}}^{(-1/2, 1/2)}(-x)=(-1)^{\frac{p-1}{2}} F\left(\frac{1-p}{2},
\frac{1+p}{2}; \frac{1}{2}; \frac{1+x}{2}\right).
\label{eq17}
\end{equation}
Setting  $x=8t-1$, by \eqref{eq09}, we obtain
\begin{equation}
(-1)^{\frac{p-1}{2}}w_{\frac{p-1}{2}}(8t-1)=\sum_{k=0}^{(p-1)/2}\binom{2k}{k}
\prod_{j=1}^k\left(1-\frac{p^2}{(2j-1)^2}\right)t^k.
\label{eq15}
\end{equation}
Now the required congruence easily follows from (\ref{A3}).
\end{proof}

\begin{proposition} \label{p1}
For any non-negative integer $n$ we have
\begin{align}
\int_0^tw_n(1-\tau^2/2)\,d\tau&=2\sum_{k=0}^{n-1}
\frac{(-1)^k v_{2k+1}(t)}{2k+1}
+\frac{(-1)^n v_{2n+1}(t)}{2n+1}, \label{intu}\\
\int_0^t \frac{(-1)^nw_n(\tau^2/2-1)-1}{\tau}\,d\tau&=\sum_{k=1}^n\frac{(-1)^{k} v_{2k}(t)}{2k}
-H_n(1). \label{intu1}
\end{align}
\end{proposition}
\begin{proof}
From (\ref{eq08b}) it follows that
$$
w_n(1-\tau^2/2)=u_{n+1}(2-\tau^2)+u_n(2-\tau^2).
$$
Recall that the generating functions of the sequences $u_n(t)$ and $v_n(t)$ have the form:
$$
U_t(z)=\sum_{n=0}^{\infty}u_n(t)z^n=\frac{z}{z^2-t z+1} \quad\text{and}\quad
V_t(z)=\sum_{n=0}^{\infty}v_n(t)z^n=\frac{2-zt}{z^2-t z+1}.
$$
Integrating corresponding power series and comparing coefficients of powers $z^{2n}$ we have
\begin{align*}
\int_0^t u_{n}(2-\tau^2)\,d\tau
&=\left[z^{2n}\right](-1)^n\int_0^t U_{2-\tau^2}(-z^2)\,d\tau
=\left[z^{2n}\right](-1)^n\int_0^t \frac{-z^2\,d\tau}{z^4+(2-\tau^2) z^2+1} \\
&=\left[z^{2n}\right](-1)^n\int_0^t
\frac{-z^2\,d\tau}{(z^2+\tau z +1)(z^2-\tau z +1)} \\
&=\left[z^{2n}\right]\frac{(-1)^n z}{2(z^2+1)}\left(\log(z^2-tz+1)-\log(z^2+tz+1)\right).
\end{align*}
Since
$$\frac{d}{dz}(\log(z^2-tz+1))=\frac{2z-t}{z^2-tz+1}=\frac{2-V_t(z)}{z}=
-\sum_{n=1}^{\infty}v_n(t)\, z^{n-1}$$
it follows that
$$
\log(z^2-tz+1)=-\sum_{n=1}^{\infty}\frac{v_n(t)}{n}\, z^n
$$
and therefore,
$$
\log(z^2-tz+1)-\log(z^2+tz+1)=-2z\sum_{n=0}^{\infty}\frac{v_{2n+1}(t)}{2n+1}\, z^{2n}.
$$
Now  formula (\ref{intu}) easily follows by virtue of the following relations:
\begin{align*}
\int_0^t u_{n}(2-\tau^2)\,d\tau&=
\left[z^{2n}\right]\frac{(-1)^{n-1}z^2}{(z^2+1)}
\sum_{n=0}^{\infty}\frac{v_{2n+1}(t)}{2n+1}\, z^{2n}
=\left[(-z)^{n-1}\right]\frac{1}{(z+1)}
\sum_{n=0}^{\infty}\frac{v_{2n+1}(t)}{2n+1}\, z^{n}\\
&=\left[z^{n-1}\right]\frac{1}{(1-z)}
\sum_{n=0}^{\infty}\frac{(-1)^n v_{2n+1}(t)}{2n+1}\, z^{n}
=\sum_{k=0}^{n-1}
\frac{(-1)^k v_{2k+1}(t)}{2k+1}.
\end{align*}
To prove (\ref{intu1}), we note that
\begin{equation}
\begin{split}
\sum_{n=0}^{\infty}\frac{u_{2n+1}(\tau)-(-1)^n}{\tau}\, z^{2n}&=\frac{1}{\tau}
\left(\frac{U_{\tau}(z)-U_{\tau}(-z)}{2z}-\frac{1}{z^2+1}\right) \\
&=\frac{z}{2(z^2+1)}\left(\frac{1}{z^2-\tau z+1}-\frac{1}{z^2+\tau z+1}\right).
\label{eq27}
\end{split}
\end{equation}
Integrating with respect to $\tau$ and resembling  coefficients of $z^{2n}$ on both sides of
(\ref{eq27}) we get
\begin{equation*}
\begin{split}
\int_0^t&\frac{u_{2n+1}(\tau)-(-1)^n}{\tau}\,d\tau=\left[z^{2n}\right]
\frac{1}{z^2+1}\left(\log(1+z^2)-\frac{\log(z^2-tz+1)+\log(z^2+tz+1)}{2}\right) \\[2pt]
&=\left[z^{2n}\right]\frac{1}{z^2+1}\Bigl(\log(1+z^2)+\sum_{n=1}^{\infty}
\frac{v_{2n}(t)}{2n}z^{2n}\Bigr)
=\sum_{k=1}^n\frac{(-1)^{k+n}v_{2k}(t)}{2k}-(-1)^nH_n(1).
\end{split}
\end{equation*}
Comparing  generating functions $W_{\tau^2/2-1}(z^2)$ and
$(U_{\tau}(z)-U_{\tau}(-z))/(2z),$ we conclude that
$$
w_n(\tau^2/2-1)=u_{2n+1}(\tau),
$$
which completes the proof.
\end{proof}
\begin{theorem} \label{t5}
Let $p$ be an odd prime and $t=a/b,$ where $a, b$ are integers coprime to $p.$
Then we have the polynomial congruences
\begin{align*}
\sum_{k=0}^{(p-3)/2}\frac{\binom{2k}{k}}{(2k+1)^2}\left(\frac{t}{4}\right)^{2k}&\equiv
\frac{(-1)^{\frac{p-1}{2}}(v_p(t)-t^p)}{tp^2}+\frac{2}{tp}\sum_{k=0}^{(p-3)/2}
\frac{(-1)^k v_{2k+1}(t)}{2k+1} \pmod{p^2}, \\
\sum_{k=1}^{(p-1)/2}\frac{\binom{2k}{k}}{k}\left(\frac{t}{4}\right)^{2k} &\equiv
4q_p(2)-2p q_p^2(2)+\sum_{k=1}^{(p-1)/2}\frac{(-1)^k v_{2k}(t)}{k} \pmod{p^2}.
\end{align*}
\end{theorem}
\begin{proof}
Let $n=(p-1)/2$. From
(\ref{eq10}) it follows that
$$
w_{n}(1-\tau^2/2)=p\sum_{k=0}^{n}\frac{\binom{2k}{k}}{2k+1}\left(\frac{\tau}{4}
\right)^{2k}
\prod_{j=0}^{k-1}\left(1-\frac{p^2}{(2j+1)^2}\right).
$$
By integrating with respect to $\tau$, and  using \eqref{CCC} for $k=n$ we obtain
\begin{align*}
\int_0^tw_{n}(1-\tau^2/2)\,d\tau&=
p\sum_{k=0}^{n}\frac{\binom{2k}{k}}{(2k+1)^2}\cdot
\frac{t^{2k+1}}{16^k}\prod_{j=0}^{k-1}\left(1-\frac{p^2}{(2j+1)^2}\right)\\
&=p\sum_{k=0}^{n-1}\frac{\binom{2k}{k}}{(2k+1)^2}\cdot
\frac{t^{2k+1}}{16^k}\sum_{j=0}^{k} (-1)^jp^{2j}\,\HH_k(\{2\}^j)+\frac{t^{p}(-1)^n}{p}.
\end{align*}
On the other hand, by (\ref{intu}), we get
$$\sum_{k=0}^{n-1}\frac{\binom{2k}{k}}{(2k+1)^2}\cdot
\left(\frac{t}{4}\right)^{2k}\sum_{j=0}^{k} (-1)^jp^{2j}\,\HH_k(\{2\}^j)=
\frac{(-1)^{n} (v_{p}(t)-t^p)}{tp^2}
+{2\over tp}\sum_{k=0}^{n-1}
\frac{(-1)^k v_{2k+1}(t)}{2k+1}.$$
and the first congruence  easily follows.
To prove the second one, we note that (\ref{eq15}) yields
$$
\frac{(-1)^{\frac{p-1}{2}} w_{\frac{p-1}{2}}(\tau^2/2-1)-1}{\tau}
=\sum_{k=1}^{(p-1)/2}\binom{2k}{k}\prod_{j=1}^k\left(1-\frac{p^2}{(2j-1)^2}\right)
\frac{\tau^{2k-1}}{16^k}.
$$
Integrating the above equality with respect to $\tau$ and applying (\ref{intu1}) we get
$$
\sum_{k=1}^{(p-1)/2}\frac{\binom{2k}{k}}{k}\left(\frac{t}{4}\right)^{2k}\prod_{j=1}^k
\left(1-\frac{p^2}{(2j-1)^2}\right)=\sum_{k=1}^{(p-1)/2}\frac{(-1)^k v_{2k}(t)}{k}
-2H_{\frac{p-1}{2}}(1),
$$
which by (vi), implies the required congruence.
\end{proof}

\section{Some numerical congruences} \label{Section4}

The special values of the finite polylogarithms $\mathcal{L}_d(x)=\sum_{k=1}^{p-1}\frac{x^k}{k^d}$
investigated in \cite[Section~4]{MT:11}) allow us to give some evaluations of the polynomial congruences established in the previous section.
As an example of what this means, in the next corollary we consider the first congruence of Theorem~\ref{t1} when $16t\in\{1,-1,2,3,1/2,4\}$. Note that the result for $t=1/16$ will be refined in the next section.

\begin{corollary} \label{c1}
Let $p>3$ be a prime. Then we have
\begin{align*}
\sum_{k=0}^{(p-3)/2}\frac{\binom{2k}{k}}{(2k+1) 4^k}&\equiv (-1)^{\frac{p+1}{2}}
\Bigl(q_p(2)-\frac{p^2}{16}B_{p-3}\Bigr) \pmod{p^3}, \\
\sum_{k=0}^{(p-3)/2}\frac{\binom{2k}{k}}{(2k+1)16^k}&\equiv
\frac{(-1)^{\frac{p+1}{2}}}{36}p^2\,B_{p-3} \pmod{p^3}, \\
\sum_{k=0}^{(p-3)/2}\frac{\binom{2k}{k}}{(2k+1)8^k}&\equiv
(-1)^{\frac{p+1}{2}}\left(\frac{2}{p}\right)
\Bigl(\frac{1}{2}\,q_p(2)-\frac{p}{8}q_p^2(2)+\frac{p^2}{16}\Bigl(q_p^3(2)-
\frac{B_{p-3}}{8}\Bigr)\Bigr) \pmod{p^3}, \\
\sum_{k=0}^{(p-3)/2}\frac{\binom{2k}{k}}{(2k+1)}\left(\frac{3}{16}\right)^k
&\equiv (-1)^{\frac{p+1}{2}}\left(\frac{3}{p}\right)
\Bigl(\frac{1}{2}q_p(3)-\frac{p}{8}q_p^2(3)+p^2\Bigl(\frac{q_p^3(3)}{16}-
\frac{B_{p-3}}{27}\Bigr)\Bigr) \pmod{p^3}, \\
\sum_{k=0}^{(p-3)/2}\frac{\binom{2k}{k}}{(2k+1)(-32)^k}&\equiv
\left(\frac{2}{p}\right)\left(2q_p(2)-p\,q_p^2(2)+\frac{p^2}{3}\left(2q_p^3(2)-
\frac{7}{32}B_{p-3}\right)\right) \pmod{p^3}.
\end{align*}
Moreover, if $p>5,$ then
$$
\sum_{k=0}^{(p-3)/2}\frac{\binom{2k}{k}}{(2k+1)(-16)^k}\equiv
q_L-\frac{p^2}{15}\left(\frac{1}{2}\,q_L^3+B_{p-3}\right) \pmod{p^3}.
$$
where $q_L=(L_p-1)/p$ is the {\sl Lucas quotient} and $L_p=v_p(1,-1)$ is the $p$th {\sl Lucas number}.
\end{corollary}
\begin{proof} Let $n=(p-1)/2$.
Setting $t=1/4$ in the first formula of  Theorem~\ref{t1} and taking into account that by (\ref{eq08}),
$w_{n}(-1)=(-1)^{n}$ we get
$$
\sum_{k=0}^{n-1}\frac{\binom{2k}{k}}{(2k+1)4^k}\equiv (-1)^{n+1}
\Bigl(q_p(2)+\frac{p^2}{16}\sum_{k=1}^{p-1}\frac{v_k(-2)}{k^3}\Bigr) \pmod{p^3}.
$$
Since $v_k(-2)=2(-1)^k,$ by (i) and (vi), we have
$$
\sum_{k=1}^{p-1}\frac{v_k(-2)}{k^3}=2\sum_{k=1}^{p-1}\frac{(-1)^k}{k^3}=
\frac{H_n(3)}{2}-2H_{p-1}(3)\equiv -B_{p-3} \pmod{p},
$$
which implies the first congruence of the corollary.

Similarly, setting $t=1/16$ and using the fact that by (\ref{eq23.55}), for prime $p>3,$
\begin{equation}
w_{\frac{p-1}{2}}(1/2)=2\sin(\pi p/6)=(-1)^{\frac{p-1}{2}}
\label{eq10.5}
\end{equation}
we get by \cite[Section~4]{MT:11},
\begin{equation*}
\sum_{k=0}^{(p-3)/2}\frac{\binom{2k}{k}}{(2k+1)16^k}
\equiv \frac{(-1)^{\frac{p+1}{2}}}{4} p^2 (\mathcal{L}_3(e^{i\pi/3})+\mathcal{L}_3(e^{-i\pi/3}))
\equiv \frac{(-1)^{\frac{p+1}{2}}}{36} p^2 B_{p-3} \pmod{p^3}.
\end{equation*}
Similarly, setting $t=-1/16$ and observing that $w_{\frac{p-1}{2}}(3/2)=L_p$
by \cite[Section~4]{MT:11}, we readily get the last congruence of the corollary.
To prove the other three congruences, we note (see \cite[Lemma~4.13]{MT:11})
that for an integer $a$ coprime to $p$,
\begin{equation}
a^{\frac{p-1}{2}}\equiv \left(\frac{a}{p}\right)\Bigl(1+\frac{1}{2}p\,q_p(a)-
\frac{1}{8}p^2q_p^2(a)+\frac{1}{16}p^3q_p^3(a)\Bigr) \pmod{p^4}.
\label{eq11}
\end{equation}
From (\ref{eq08}) and (\ref{eq23.55}) we easily find that
\begin{equation}
w_{\frac{p-1}{2}}(0)=\cos(\pi(p-1)/4)+\sin(\pi(p-1)/4)=\sqrt{2}\sin(\pi p/4)
=(-1)^{\frac{p-1}{2}}\left(\frac{2}{p}\right),
\label{eq12}
\end{equation}
\begin{equation}
w_{\frac{p-1}{2}}(-1/2)=\cos(\pi(p-1)/3)+\frac{1}{\sqrt{3}}\sin(\pi(p-1)/3)
=\frac{2}{\sqrt{3}}\sin(\pi p/3)=(-1)^{\frac{p-1}{2}}\left(\frac{3}{p}\right),
\label{eq13}
\end{equation}
and
\begin{equation}
w_{\frac{p-1}{2}}(5/4)=2^{\frac{p+1}{2}}-2^{\frac{1-p}{2}}.
\label{eq13.56}
\end{equation}
Now setting consequently $t=1/8,$  $t=3/16,$ and $t=-1/32$ in Theorem~\ref{t1} by (\ref{eq11})--(\ref{eq13.56}),
\cite[Section~4]{MT:11} and the formula
$$
2^{\frac{1-p}{2}}\equiv \left(\frac{2}{p}\right)\Bigl(1-\frac{1}{2}\,p\,q_p(2)+\frac{3}{8}\,p^2q_p^2(2)
-\frac{5}{16}\,p^3q_p^3(2)\Bigr) \pmod{p^4},
$$
we get the desired congruences.
\end{proof}
By the same way, from Theorem~\ref{t1} and \cite[Section~4]{MT:11} we get the following corollary.
\begin{corollary}\label{corfour}
Let $p>3$ be a prime. Then
\begin{align*}
\sum_{k=0}^{(p-1)/2}\frac{\binom{2k}{k}}{16^k} & \equiv
\left(\frac{3}{p}\right)+\left(\frac{1}{p}\right) \frac{p^2}{24} B_{p-2}\Bigl(\frac{1}{3}\Bigr)
\pmod{p^3},  \\
\sum_{k=0}^{(p-1)/2}\binom{2k}{k}\left(\frac{3}{16}\right)^k & \equiv
1+\left(\frac{-3}{p}\right) \frac{p^2}{12} B_{p-2}\Bigl(\frac{1}{3}\Bigr)
\pmod{p^3}.
\end{align*}
\end{corollary}

Note that the first congruence modulo $p^2$ in Corollary~\ref{corfour} appeared in \cite[Corollary~1.1]{Sunzwfib}.

Recall that the Fibonacci numbers $\{F_n\}_{n\ge 0}$ and the Lucas numbers
$\{L_n\}_{n\ge 0}$
are defined by $F_n=u_n(1,-1)$ and  $L_n=v_n(1,-1)$ for all non-negative integers $n$.

In the next theorem we confirm two conjectures raised in \cite[A90]{Sunzwopen}.
\begin{theorem} \label{t2}
Let $p\ne 2, 5$ be a prime. Then we have
\begin{equation*}
\begin{split}
\sum_{k=0}^{(p-3)/2}\frac{\binom{2k}{k}F_{2k+1}}{(2k+1)16^k}&\equiv
(-1)^{\frac{p+1}{2}}\;\frac{F_p-\left(\frac{p}{5}\right)}{p}
\pmod{p^2}, \\
\sum_{k=0}^{(p-3)/2}\frac{\binom{2k}{k}L_{2k+1}}{(2k+1)16^k}&\equiv
(-1)^{\frac{p+1}{2}}\;\frac{L_p-1}{p}
\pmod{p^2}.
\end{split}
\end{equation*}
\end{theorem}
\begin{proof} We first show that for an odd prime $p\ne 5$  and $\phi_{\pm}=(1\pm\sqrt{5})/2,$ we have
\begin{align}\label{w1}
\phi_{+}\cdot w_{\frac{p-1}{2}}(\phi_{-}/2)-
\phi_{-}\cdot w_{\frac{p-1}{2}}(\phi_{+}/2)&=(-1)^{\frac{p-1}{2}}
\left(\frac{p}{5}\right)\sqrt{5}, \\\label{w2}
\phi_{+}\cdot w_{\frac{p-1}{2}}(\phi_{-}/2)+
\phi_{-}\cdot w_{\frac{p-1}{2}}(\phi_{+}/2)&=(-1)^{\frac{p-1}{2}}.
\end{align}
Since $\arccos(\phi_{+}/2)=\pi/5$ and $\arccos(\phi_{-}/2)=3\pi/5,$ by (\ref{eq23.55}), we get
\begin{align*}
\phi_{+}\cdot w_{\frac{p-1}{2}}(\phi_{-}/2)\mp\phi_{-}\cdot
w_{\frac{p-1}{2}}(\phi_{+}/2) \\
=2\Bigl(\cos(\pi/5)\cos(3\pi(p-1)/10)&\mp\cos(3\pi/5)\cos(\pi(p-1)/10)\Bigr) \\
+\frac{\sqrt{5}}{2}\left(\frac{\sin(3\pi(p-1)/10)}{\sin(3\pi/5)}\right.
&\pm\left.\frac{\sin(\pi(p-1)/10)}{\sin(\pi/5)}\right).
\end{align*}
Considering primes $p$ modulo $10$ and using the well-known fact that
\begin{equation*}
\left(\frac{p}{5}\right)=\begin{cases}
\,\,\, \,1 & \quad\text{if}\quad p\equiv\pm 1 \pmod{5} \\
-1 & \quad\text{if}\quad p\equiv\pm 2 \pmod{5}
\end{cases}
\end{equation*}
we easily derive the result. Since
$$
F_{2k+1}=\frac{1}{\sqrt{5}}\left(\phi_{+}^{2k+1}-
\phi_{-}^{2k+1}\right),
$$
from Theorem~\ref{t1} we readily find that
\begin{equation*}
\begin{split}
\sum_{k=0}^{(p-3)/2}\frac{\binom{2k}{k} F_{2k+1}}{(2k+1)16^k}&\equiv
\frac{1}{p\sqrt{5}}\left(\phi_{+}\cdot w_{\frac{p-1}{2}}(
\phi_{-}/2)+(-1)^{\frac{p+1}{2}}\phi_{+}^p \right.\\[5pt]
&\left.-\phi_{-}\cdot w_{\frac{p-1}{2}}(\phi_{+}/2)
-(-1)^{\frac{p+1}{2}}\phi_{-}^p
\right) \pmod{p^2}.
\end{split}
\end{equation*}
Now by \eqref{w1}, we get the first congruence of Theorem~\ref{t2}. Analogously,
for the Lucas numbers we have
$$
L_{2k+1}=\phi_{+}^{2k+1}+\phi_{-}^{2k+1}
$$
and therefore, by Theorem~\ref{t1} and \eqref{w2}, we obtain
\begin{equation*}
\begin{split}
\sum_{k=0}^{(p-3)/2}\frac{\binom{2k}{k}L_{2k+1}}{(2k+1)16^k}&\equiv
\frac{1}{p}\left(\phi_{+}\cdot w_{\frac{p-1}{2}}(\phi_{-}/2)
+(-1)^{\frac{p+1}{2}}\phi_{+}^p\right. \\
&\left.+
\phi_{-}\cdot w_{\frac{p-1}{2}}(\phi_{+}/2)
+(-1)^{\frac{p+1}{2}}\phi_{-}^p\right)
=\frac{(-1)^{\frac{p+1}{2}}}{p}(L_p-1) \pmod{p^2},
\end{split}
\end{equation*}
which completes the proof.
\end{proof}
\begin{corollary}
Let $p$ be an odd prime. Then we have
\begin{align*}
\sum_{k=0}^{(p-3)/2}\frac{\binom{2k}{k}}{(2k+1)^2 4^k}& \equiv (-1)^{\frac{p+1}{2}}
\Bigl(\frac{1}{2}q_p^2(2)-\frac{1}{3}p\,q_p^3(2)-\frac{1}{16}pB_{p-3}\Bigr) \pmod{p^2}, \\
\sum_{k=1}^{(p-1)/2}\frac{\binom{2k}{k}}{k 4^k}& \equiv 2q_p(2)-p\,q_p^2(2)+(-1)^{\frac{p+1}{2}}2pE_{p-3} \pmod{p^2}.
\end{align*}
\end{corollary}
\begin{proof}
Setting $t=2$ in Theorem~\ref{t5} and noting that $v_n(2)=2$ for all $n\in {\mathbb N},$ we get
\begin{equation*}
\sum_{k=0}^{(p-3)/2}\frac{\binom{2k}{k}}{(2k+1)^2 4^k}\equiv \frac{(-1)^{\frac{p+1}{2}}q_p(2)}{p}
+\frac{2}{p}\sum_{k=0}^{(p-3)/2}\frac{(-1)^k}{2k+1} \pmod{p^2}.
\end{equation*}
Now the required congruence easily follows from Lemmas~\ref{l3} and \ref{lemmaHH}, and formulae (\ref{C2}) and (vi).
Similarly, we have
\begin{equation*}
\sum_{k=1}^{(p-1)/2}\frac{\binom{2k}{k}}{k 4^k}\equiv
4q_p(2)-2p\,q_p^2(2)+2\sum_{k=1}^{(p-1)/2}\frac{(-1)^k}{k} \pmod{p^2}.
\end{equation*}
Observing that by \cite[Corollary~3.3]{Sunzh:08} and (vi),
\begin{equation*}
\sum_{k=1}^{(p-1)/2}\frac{(-1)^k}{k}
=H_{\lfloor p/4\rfloor}-H_{(p-1)/2}\equiv
-q_p(2)+\frac{1}{2}p\,q_p^2(2)+(-1)^{\frac{p+1}{2}}pE_{p-3} \pmod{p^2},
\end{equation*}
we conclude the proof.
\end{proof}

\section{Two remarkable special cases} \label{Section5}
By revisiting the combinatorial proof (due to D. Zagier) presented  in \cite[Section 5]{Le:81}, it is easy to obtain
the finite versions of  identities (\ref{id1}), (\ref{id2}) in the following form: if $r$ is a positive odd integer then
\begin{align}\label{idodd}\nonumber
\sum_{k=0}^{n-1}{{2k\choose k}\over 16^k}&
\left(\sum_{j=0}^{{r-1\over 2}}{(-1)^j\HH_k(\{2\}^j)\over (2k+1)^{r-2j}}
-{(-1)^{{r-1\over 2}}\over 4}\cdot{\HH_k(\{2\}^{{r-1\over 2}})\over (2k+1)}\right)\\
&=\sum_{k=0}^{n-1}{(-1)^k\over (2k+1)^r}+{(-1)^{{r-1\over 2}}\over 4}
\sum_{k=0}^{n-1}{{2k\choose k}\over 16^k{n+k\choose 2k+1}}\cdot{(-1)^{n-k}\HH_k(\{2\}^{{r-1\over 2}})\over (2k+1)},
\end{align}
and if $r$ is a positive even integer then
\begin{align}\label{ideven}\nonumber
\sum_{k=0}^{n-1}{{2k\choose k}\over (-16)^k}&
\left(\sum_{j=0}^{{r\over 2}-1}{(-1)^j\HH_k(\{2\}^j)\over (2k+1)^{r-2j}}
+{(-1)^{{r\over 2}-1}\over 4}\cdot{\HH_k(\{2\}^{{r\over 2}-1})\over (2k+1)^2}\right)\\
&=\sum_{k=0}^{n-1}{1\over (2k+1)^r}+{(-1)^{{r\over 2}-1}\over 4}
\sum_{k=0}^{n-1}{{2k\choose k}\over (-16)^k{n+k\choose 2k+1}}\cdot{\HH_k(\{2\}^{{r\over 2}-1})\over (2k+1)^2}.
\end{align}
\begin{theorem}\label{TM} For any prime $p>5$ we have
\begin{align}\label{mc1}
\sum_{k=0}^{{p-3\over 2}}{{2k \choose k}\over 16^k(2k+1)}&\equiv
(-1)^{{p-1\over 2}}\left({H_{p-1}(1)\over 12}+{3\over 160}\,p^4 B_{p-5}\right)
\pmod{p^5},\\
\label{mc2}
\sum_{k=0}^{{p-3\over 2}}{{2k \choose k}\over (-16)^k(2k+1)^2}&\equiv
{H_{p-1}(1)\over 5p}+{7\over 20}\,p^3 B_{p-5}
\pmod{p^4}.
\end{align}
\end{theorem}
\begin{proof}
According to \eqref{CCC}, it follows that
\begin{align}\label{conbin}\nonumber
{{2k\choose k}\over (-16)^k{n+k\choose 2k+1}}&={(2k+1){2k\choose k}\over (-16)^k(n-k){n+k\choose 2k}}
\equiv -{2\over \left(1-{p\over 2k+1}\right)\left(1-p^2\HH_k(2)+p^4\HH_k(2,2)\right)}\\\nonumber
&\equiv -2\left(1+{p\over 2k+1}+
\left({1\over (2k+1)^2}+\HH_k(2)\right)\,p^2+\left({1\over (2k+1)^3}+{\HH_k(2)\over 2k+1}\right)\,p^3
\right.\\
&\qquad+\left.
\left({1\over (2k+1)^4}+{\HH_k(2)\over (2k+1)^2}+\HH_k(4)+\HH_k(2,2)\right)\,p^4
\right)\pmod{p^5}.
\end{align}
Therefore, by \eqref{idodd} for $r=1$, \eqref{conbin} and \eqref{alts}, we get
\begin{align}\label{a2}\nonumber
{3(-1)^n\over 2}\sum_{k=0}^{n-1}{{2k\choose k}\over 16^k(2k+1)}&\equiv
-2p\HH_{n}(2) -p^2(\HH_{n}(3)+2\HH_{n}(2,1))-p^3\HH_{n}(4)\\
&\qquad-p^4(\HH_{n}(5)+\HH_{n}(2,3)+\HH_{n}(4,1)))\pmod{p^5}.
\end{align}
Now, by Lemma~\ref{lemmaHH} we can replace the terms $\HH_n$ with the corresponding expressions involving $H_{p-1}$ and $H_n$. So the right-hand side of \eqref{a2} becomes
\begin{align*}
&-2pH_{p-1}(2)-p^2H_{p-1}(3)-p^3H_{p-1}(4)-p^4H_{p-1}(5)\\
&\qquad +{p\over 2}H_n(2)+{p^2\over 8}(H_n(3)+2H_n(1,2))
+{p^3\over 16}\left(H_n(4)+4H_n(1,3)+2H_n(2,2)\right)\\
&\qquad-{p^4\over 32}\left(H_n(5)+4H_n(2,3)+3H_n(3,2)+7H_n(1,4)\right) \pmod{p^5}.
\end{align*}
Since
$$H_n(2,2)={1\over 2}\left( H_n(2)^2-H_n(4)\right)\equiv -{31\over 5}\, pB_{p-5} \pmod{p^2},$$
by (i), (ii), and (vi), the above expression simplifies to
$$-2p H_{p-1}(2)+{1\over 2}\,pH_n(2)+{1\over 8}\,p^2 H_n(3)
+{1\over 4}\,p^2\left(H_n(1,2)+pH_n(1,3)\right)+{513\over 320}\,p^4 B_{p-5}
\pmod{p^5}.$$
Finally, we apply \eqref{Hp2}, \eqref{H2}, \eqref{H3} and \eqref{unoduepiuunotre}
to obtain
$${H_{p-1}(1)\over 8}+{9\over 320}\,p^4 B_{p-5}\pmod{p^5}$$
which concludes our proof of \eqref{mc1}.

As regards \eqref{mc2}, by \eqref{ideven} for $r=2$, and \eqref{conbin}, we have
\begin{align}\label{a1}\nonumber
{5}\sum_{k=0}^{n-1}{{2k\choose k}\over (2k+1)^2(-16)^k}&\equiv
4\HH_{n}(2)-2(\HH_{n}(2)+p\HH_{n}(3)+p^2(\HH_{n}(4)+\HH_{n}(2,2))\\
&\qquad +p^3(\HH_{n}(5)+\HH_{n}(2,3)))\pmod{p^4}.
\end{align}
As before, after replacing the terms $\HH_n$, the right-hand side of \eqref{a1} becomes
\begin{align*}
&2H_{p-1}(2)-2pH_{p-1}(3)-2p^2H_{p-1}(4)-2p^3H_{p-1}(5)\\
&\qquad -{1\over 2}H_n(2)+{p\over 4}H_n(3)+{p^2\over 8}\left(H_n(4)- H_n(2,2)\right)\\
&\qquad+{p^3\over 16}\left(H_n(5)-2H_n(2,3)-H_n(3,2)\right)  \pmod{p^4}.
\end{align*}
By (i), (ii), and (vi), it simplifies to
$$2H_{p-1}(2)-{1\over 2}H_n(2)+{1\over 4}p H_n(3) +{9\over 4}p^3B_{p-5}     \pmod{p^4}.$$
By using \eqref{Hp2}, \eqref{H2} and \eqref{H3}, we get
$${H_{p-1}(1)\over p}+{7\over 4}\,p^3B_{p-5} \pmod{p^4}$$
and \eqref{mc2} is established.
\end{proof}
Note that the congruence (\ref{mc1}) as well as (\ref{mc2}) modulo  $p^3$ in Theorem~\ref{TM}
 were first conjectured by Z.~W.~Sun in \cite[Conjecture~5.1]{Sunzw:11}.
From (\ref{mc1}) and (\ref{A}) we easily get the following congruence, which was proposed in \cite[A32]{Sunzwopen}.
\begin{corollary}
Let $p>5$ be a prime. Then
$$
\sum_{k=0}^{(p-3)/2}\frac{\binom{2k}{k}\overline{H}_k(2)}{16^k(2k+1)}\equiv (-1)^{\frac{p-1}{2}}
\frac{H_{p-1}(1)}{12p^2} \pmod{p^2}.
$$
\end{corollary}
Note that the value of the corresponding infinite series is known (see \cite[p.~230--231]{HH:10})
$$
\sum_{k=1}^{\infty}\frac{\binom{2k}{k}\overline{H}_k(2)}{16^k(2k+1)}=\frac{\pi^3}{648}.
$$


\vspace{0.2cm}

{\bf\small Acknowledgement.}
The first and second authors wish to thank
the Abdus Salam International Centre for Theoretical Physics (ICTP), Trieste, Italy,
 and, in particular,  the Head of the Mathematics Section of the  ICTP,
Professor Ramadas Ramakrishnan for the hospitality and excellent working conditions during their summer visit in 2011 when  part of this work was done.

\end{document}